%% file: document.tex
\documentclass[a4paper,twoside]{amsart}
\usepackage[utf8]{inputenc}
\usepackage{amsmath}
\usepackage{amsfonts}
\usepackage{amsthm}
\usepackage{amssymb}
\usepackage{tikz}
\usepackage{mathrsfs}
\usepackage[all,cmtip]{xy}

\title[The Binomial Theorem and motivic classes of tori]{The Binomial Theorem
and\\
motivic classes of universal quasi-split tori}
\author{Daniel Bergh}
\date{}

\include{declarations}

\begin{document}
\begin{abstract}
Taking symmetric powers of varieties can be seen as a functor
from the category of varieties to the category of varieties
with an action by the symmetric group. We study a corresponding
map between the Grothendieck groups of these categories. In
particular, we derive a binomial formula and use it to give
explicit expressions for the classes of universal quasi-split
tori in the equivariant Grothendieck group of varieties.
\end{abstract}
\maketitle
\input{intro}
\input{g-rings}

\input{varieties}
\input{tori}
\bibliographystyle{alpha}
\bibliography{main}
\end{document}

%% file: declarations.tex
\theoremstyle{plain}
\newtheorem{theorem}{Theorem}

\newtheorem*{thmBinomial}{The Binomial Theorem}

\newtheorem{prop}[theorem]{Proposition}
\newtheorem{lemma}[theorem]{Lemma}

\theoremstyle{definition}
\newtheorem*{definition}{Definition}

\theoremstyle{remark}

\newtheorem*{remark}{Remark}

\DeclareMathOperator{\Spec}{Spec}

\DeclareMathOperator{\K0}{K_0}

\DeclareMathOperator{\Sd}{Sd}
\DeclareMathOperator{\Conf}{Conf}
\let\Im\relax
\DeclareMathOperator{\Im}{Im}

\newcommand{\pK}[0]{\ensuremath{\mathscr{K}_0}}
\newcommand{\pU}[0]{\ensuremath{\mathscr{U}_0}}
\newcommand{\Bu}[0]{\ensuremath{\mathrm{A}}}
\newcommand{\pBu}[0]{\ensuremath{\mathscr{A}}}

\newcommand{\li}[0]{\ensuremath{\Lambda}}
\newcommand{\rli}[0]{\ensuremath{\tilde{\Lambda}}}

\newcommand{\cat}[1]{\ensuremath{\mathrm{#1}}}

\newcommand{\catset}[1]{\cat{Set}}

\newcommand{\term}[1]{{\em #1}}

\newcommand{\ZZ}[0]{\ensuremath{\mathbb{Z}}}
\newcommand{\NN}[0]{\ensuremath{\mathbb{N}}}

\newcommand{\LL}[0]{\ensuremath{\mathbb{L}}}

\newcommand{\GG}[0]{\ensuremath{\mathbb{G}}}
\newcommand{\GGm}[0]{\ensuremath{\mathbb{G}_\mathrm{m}}}

\renewcommand{\AA}[0]{\ensuremath{\mathbb{A}}}
\newcommand{\BB}[0]{\ensuremath{\mathrm{B}}}

\newcommand{\Ind}[2]{\ensuremath{\mathrm{Ind}_{#1}^{#2}}}
\newcommand{\Res}[2]{\ensuremath{\mathrm{Res}_{#1}^{#2}}}
\newcommand{\Lef}[0]{\ensuremath{\Lambda}}
\newcommand{\rLef}[0]{\ensuremath{\widetilde{\Lambda}}}

\newcommand{\lcat}[1]{\ensuremath{\mathrm{#1}}}

\newcommand{\cVar}[0]{\ensuremath{\lcat{Var}}}
\newcommand{\cgVar}[1]{\ensuremath{{#1}\mbox{-}\lcat{Var}}}

\newcommand{\todo}[1]{
}

%% file: intro.tex
\section*{Introduction}
Given a variety $X$ over some field $k$ and a finite $G$-set $S$ for a finite group
$G$ we may form the power $X^S$. As a variety this is simply the $n$-fold cartesian
product of $X$ with itself, where $n$ is the cardinality of $S$. The action of
$G$ on $S$ induces an action of $G$ on $X^S$ by permuting the factors. This gives a
\term{power functor}
$$
(-)^S\colon \cVar_k \to \cgVar{G}_{k}
$$
from the category $\cVar_k$ of varieties over $k$ to the category $\cgVar{G}_{k}$
of $G$-equi\-variant varieties over $k$.

By using an exponential function defined on the Grothendieck group of varieties,
we show that the power functor extends to a map between the corresponding
Grothendieck rings in a natural way. The construction of the exponential
function was originally developed by Bouc in the context of Burnside rings
in \cite{bouc1992}.

Denote the set with $n$ elements with its natural action of the
symmetric group $\Sigma_n$ by $[n]$.
\begin{thmBinomial}
\label{thm-binomial}
The symmetric power functor $(-)^{[n]}$ extends to a map
between Grothendieck rings
$
(-)^{[n]}\colon\K0(\cVar_k) \to \K0(\cgVar{\Sigma_n}_k)
$
such that the relation
$$
(x + y)^{[n]} = \sum_{i + j = n}
\Ind{\Sigma_i\times\Sigma_j}{\Sigma_n}\left(x^{[i]}\boxtimes y^{[j]}\right)
$$
holds for arbitrary elements $x, y$ in $\K0(\cVar_k)$. Furthermore, we
have the relation $\{E\}^{[n]} = \LL^{dn}\{X\}^{[n]}$ for any rank $d$
vector bundle $E \to X$.
\end{thmBinomial}

Our main application of this formula is to obtain an explicit expression for
the class of the $G$-variety $\GGm^S$ in the Grothendieck group
$\K0(\cgVar{G}_k)$ of $G$-equivariant varieties. The Binomial Theorem allows us
to reduce the problem of deriving such a formula to combinatorial calculations in the Burnside
ring. Our main result is the following:

\begin{theorem}
\label{thm-equiv-torus}
Let $G$ be a finite group and $S$ a finite $G$-set with $n$ elements. Then the
class of $\GGm^S$ is equal to
$$
\sum_{i = 0}^n (-1)^i\lambda^i(S)\LL^{n-i}
$$
in $\K0(\cgVar{G}_k)$.
\end{theorem}
\noindent
Here $\lambda^i$ denotes the natural lambda ring operations on
$\K0(\cgVar{G}_k)$.

The significance of the $G$-variety $\GGm^S$ is that it can be regarded as a
\term{universal quasi-split torus}. Recall that a \term{split torus} over a
field $k$ is a finite product of $n$ copies of the group $\GGm$. In general, an
algebraic group $T$ over $k$ is a torus if the base change $T_{\bar{k}}$ to an
algebraic closure $\bar{k}$ of $k$ is a split torus. We call $T$
\term{quasi-split} if it is the Weil-restriction
of the group $\GGm$ along a map $\Spec L \to \Spec k$, where $L$ is a separable
algebra over $k$. In particular, the rational points of such a torus is simply
the group $L^\times$ of units in $L$.

Every rank $n$ quasi-split torus can be obtained from $\GGm^{[n]}$ via descent
in the following way. Given a separable $k$-algebra $L$ of degree $n$, we can form
the configuration space $\Conf^n(\Spec L)$, which is a $\Sigma_n$-torsor over
$k$. Descent along this torsor gives a functor $\cgVar{\Sigma_n}_k \to \cVar_k$,
and $L^\times$ is the image of $\GG_m^{[n]}$ under this functor. The functor
induces a lambda-ring homomorphism $\K0(\cgVar{\Sigma_n}_k) \to \K0(\cVar_k)$,
and we obtain the following result from \cite{rokaeus2007} as a corollary to our
main theorem.

\begin{theorem}[Rökaeus]
\label{thm-field-torus}
Let $L$ be a finite separable algebra of degree $n$ over the field $k$.
Then the class of the torus $L^\times$ equals
$$
\sum_{i = 0}^{n} (-1)^i\lambda^i(\Spec(L))\LL^{n-i}
$$
in $\K0(\cVar_k)$.
\end{theorem}

The true universal nature of $\GG_m^{[n]}$ is best understood in the language
of algebraic stacks. The stack quotient $[\GG_m^{[n]}/\Sigma_n]$ is a group
object over the classifying stack $\BB \Sigma_n$. A $\Sigma_n$-torsor over $k$
corresponds to a map $\Spec k \to \BB \Sigma_n$, and $L^\times$ is simply the
pull-back of $[\GG_m^{[n]}/\Sigma_n]$ along the map corresponding to
$\Conf^n(\Spec L)$. Since the category of $G$-equivariant varieties is
equivalent to the category of varieties over the base $\BB G$, we can reformulate
Theorem~\ref{thm-equiv-torus} as follows:

\begin{theorem}
\label{thm-stack-torus}
Let $G$ be a finite group and $S$ a finite $G$-set with $n$ elements. Then the
class of $[\GGm^S/G]$ is equal to
$$
\sum_{i = 0}^n (-1)^i\lambda^i([S/G])\LL^{n-i}
$$
in $\K0(\cVar_{\BB G})$.
\end{theorem}

The need for explicit expressions for classes of universal quasi-split tori arose
in \cite{bergh2014}, where I computed the classes of certain classifying stacks
in the Grothendieck group of stacks. It turns out that classifying stacks of
monomial matrices are intimately related to such tori. In this article, the bigger
class of stably rational tori is also studied.

It should be pointed out that the proof of Theorem~\ref{thm-field-torus} by
Rökaeus could easily be adapted to the universal setting. However,
using the Binomial Theorem has its merits in its conceptual clearness. For
instance, the method could also be used to compute the class of the Weil
restriction along a finite étale map of any variety whose class can be expressed
as a polynomial in $\LL$. It should also be pointed out that in this article we
use a considerably down-powered version of Bouc's machinery. A remark about
this concludes the last section. As a final remark, it should be said that the
choice of working over a field is done purely for psychological reasons; we
could equally well have worked in the relative setting over an arbitrary
scheme or algebraic stack.

The rest of the article is organised into three sections. In the first, we
review some basic facts about equivariant Grothendieck rings of varieties
and Burnside rings. In the second, we construct the exponential function and
prove the Binomial Theorem. In the last, we relate the exponential
function to the corresponding construction for Burnside rings and give the
proof of Theorem~\ref{thm-equiv-torus}.

\subsection*{Acknowledgements}
I would like to thank my advisor, David Rydh, for suggesting several
improvements of this text.
%

%% file: g-rings.tex
\section*{Preliminaries}
By \term{variety} in this paper, we will always mean a finitely presented
algebraic space over a field $k$, and we denote the category of such object
by $\cVar_k$. Given a finite group $G$, we denote by $\cgVar{G}_k$ the category
whose objects are varieties endowed with a $G$-action and whose morphisms are
$G$-equivariant maps.

\subsection*{Equivariant Grothendieck rings}
Given a finite group $G$, we define the \term{Grothendieck group of
$G$-equivariant varieties} $\K0(\cgVar{G}_k)$ as the free group on the
isomorphism classes $\{X\}$ of objects $X$ in $\cgVar{G}_k$, subject to the
following relations:
\begin{enumerate}
\item[{\bf R1}] If $Z \subset X$ is a closed, $G$-invariant subvariety, then
$\{X\} = \{X\setminus Z\} + \{Z\}$.
\item[{\bf R2}] If $E \to X$ is a $G$-equivariant vector bundle of rank $d$,
then $\{E\} = \{\AA^d\times X\}$.
\end{enumerate}
The class of the affine line $\AA^1_k$ with the trivial $G$-action is called
the \term{Lefschetz class} and is denoted by $\LL$.

The group $\K0(\cgVar{G}_k)$ has a structure of commutative ring with identity
induced by the categorical product in $\cgVar{G}_k$. We sometimes call
$\K0(\cgVar{G}_k)$ the \term{$G$-equivariant Grothendieck ring} if we want to
emphasise its multiplicative structure. We have an obvious ring homomorphism
from $\ZZ[\LL]$ to $\K0(\cgVar{G}_k)$, which turns out to be injective, and
sometimes, we prefer to think of the group $\K0(\cgVar{G}_k)$ as a
$\ZZ[\LL]$-module.

The usual Grothendieck ring of varieties, which we denote by $\K0(\cVar_k)$, is
recovered from the definition above if we take $G$ to be the trivial group. In
this case, the relation R2 is redundant.

Equivariant Grothendieck rings for different Grothendieck groups are related 
to each other by various operations. Given finite groups $G$ and $H$, we
have a functor $\boxtimes\colon\cgVar{G}_k \times \cgVar{H}_k \to \cgVar{G\times
H}_k$ taking a pair $(X, Y)$ to $X\times Y$ endowed with its natural
$G\times H$-action. This functor induces a corresponding
\term{outer product}
$$
\boxtimes\colon\K0(\cgVar{G}_k) \times \K0(\cgVar{H}_k) \to \K0(\cgVar{G\times
H}_k)
$$
on the Grothendieck groups. It is easy to see that this operation is
$\ZZ[\LL]$-bilinear.

Assume that we have an injective group homomorphism $H \to G$. Then there is
natural functor $\cgVar{G}_k \to \cgVar{H}_k$ given by restricting the group
action along $H \to G$. This functor has a left adjoint given by induction, and
the adjoint pair induces a pair of maps
$$
\Res{H}{G}\colon\K0(\cgVar{G}_k) \to \K0(\cgVar{H}_k),\qquad
\Ind{H}{G}\colon\K0(\cgVar{H}_k)\to \K0(\cgVar{G}_k)
$$
between the corresponding equivariant Grothendieck rings. The map $\Res{H}{G}$
is a ring homomorphism and $\Ind{H}{G}$ is a group homomorphism. In addition,
the map $\Ind{H}{G}$ is $\K0(\cgVar{G}_k)$-linear in the sense that
it satisfies the \term{projection formula}
$$
y\cdot\Ind{H}{G}(x) = \Ind{H}{G}\left(\Res{H}{G}(y)\cdot x\right)
$$
for $x \in \K0(\cgVar{H}_k)$ and $y \in \K0(\cgVar{G}_k)$. By applying this to
the class $\LL$, we see that $\Ind{H}{G}$ is a $\ZZ[\LL]$-module homomorphism.

The equivariant Grothendieck rings are also endowed with lambda ring
structures, and the restriction maps are compatible with these structures.
In this article, we will only need the lambda operations for classes which
come from the Burnside ring, as will be described below.

The equivariant Grothendieck ring $\K0(\cVar_k)$ can also be thought of
as the Grothendieck ring for the category $\cVar_{\BB G}$. This is the
category of algebraic stacks over the classifying stack $\BB G$ whose
structure morphisms are representable by finitely presented algebraic spaces.
In this context, the maps $\Res{H}{G}$ and $\Ind{H}{G}$ are induced by
the functors $f^\ast$ and $f_!$, where $f\colon\BB H \to \BB G$ is the
map between the classifying stacks induced by the inclusion of groups.
The functor $f^\ast$ is simply pull-back along $f$, and $f_!$ is obtained
by post composition with $f$. Although conceptually appealing, this
point of view will not be used in the rest of the exposition.

\begin{remark}
The admittedly somewhat unusual use of the term variety requires some
explanation. On one hand, the class of algebraic spaces are much better behaved
with respect to quotients by finite groups than traditional varieties.
For instance, the functor $\cgVar{G}_k \to \cVar_k$ induced by descent
along a $G$-torsor, which was mentioned in the introduction, does not exist
for classical varieties. On the other hand, the rings $\K0(\cVar_k)$ which we
get if we use either the classical definition of variety, or our definition,
are canonically isomorphic. Hence there seems to be little motivation for
using the term \term{Grothendieck group of algebraic spaces} rather than the
much more well-established \term{Grothendieck group of varieties}.

If one prefers to work with classical varieties, one can follow
the approach used by Heinloth-Bittner in \cite{bittner2004},
and define a $G$-variety as a classical variety endowed with
a \term{good} $G$-action. The results in this article regarding
the class of universal quasi-split tori would be equally valid
in that setting.
\end{remark}

\subsection*{Burnside rings}
The \term{Burnside ring} $\Bu(G)$ for a finite group $G$ is defined as the
Grothendieck ring of the monoid of isomorphism classes of finite
$G$-sets. The multiplication and addition in this ring are induced by
the categorical product and coproduct respectively. The ring $\Bu(G)$
is also endowed with naturally defined lambda operators $\lambda^i$, giving
it the structure of a (non-special) lambda ring. For an introduction
on the Burnside rings and its lambda structure, see \cite{knutson1973}.

Just as with the equivariant Grothendieck rings, we can relate Burnside rings
for different groups $G$ and $H$ with the bilinear operation
$\boxtimes\colon \Bu(G)\times\Bu(H) \to \Bu(G\times H)$. Given an injective
group homomorphism $H \to G$, we also have the maps $\Res{H}{G}$ and $\Ind{H}{G}$
with the former being a ring homomorphism and the latter being a group
homomorphism. The pair also satisfies the projection formula.

We will use the Burnside rings for computing certain coefficients in
$\K0(\cVar_k)$. There is a natural ring homomorphism $\Bu(G) \to
\K0(\cgVar{G}_k)$ which is induced by the functor taking a $G$-set to the same
$G$-set viewed as a $G$-variety. The ring homomorphism respects the lambda
structure, and is compatible with the operations $\boxtimes$, $\Ind{H}{G}$ and
$\Res{H}{G}$ in the obvious sense.

As a computational aid when performing calculations in the Burnside ring, we use
the \term{Lefschetz invariant}. A self contained introduction to these techniques
is given in \cite[§4]{bouc2000}, and will follow the notation used in this source.
The Lefschetz invariant is defined on the class of $G$-posets. By a $G$-poset,
we mean a finite partially ordered set endowed with a $G$-action, The $G$-action
is required to respect the ordering in the sense that $gx \leq gy$ for
all related pairs $x \leq y$ in the poset and all group elements $g\in G$.
Given a $G$-poset $P$, we denote the set of $i$-chains $a_0 < \ldots < a_i$ by $\Sd_i$.
This set has a natural $G$-action, and we can consider its class in the
Burnside ring, which we also denote by $\Sd_i$. The \term{Lefschetz invariant}
of a $G$-poset $P$ is defined as the element
$$
\Lef_P = \sum_{i \geq 0} (-1)^i\Sd_i P
$$
in $\Bu(G)$. Sometimes, it is more convenient to use the \term{reduced
Lefschetz invariant}, which is defined as $\rLef_P = \Lef - 1$, where $1$
denotes the class of the trivial $G$-representation. A fundamental fact is that
every element in $\Bu(G)$ can be represented as $\Lef_P$ for some $G$-poset $P$
\cite[Lemme 2]{bouc1992}. In particular, we have $\Lef_{S} = \{S\}$ for
every $G$-set $S$, where we on the left hand side regard $S$ as a
discrete $G$-poset.

A morphism between two $G$-posets is defined as a $G$-equivariant,
order preserving map. There is a concept of \term{homotopy} for such
morphisms. We say that the morphisms $f,g\colon P\to Q$ are \term{homotopic}
if they are comparable in the pointwise ordering of functions.
This gives a notion of \term{homotopy equivalence} of $G$-posets. We
have the relation  $\Lef_P = \Lef_Q$ in the Burnside ring provided
that $P$ and $Q$ are homotopy equivalent. In particular, the reduced
Lefschetz invariant for a poset $P$ vanishes if $P$ has a maximal or
minimal element. Indeed, this would imply that $P$ is homotopy equivalent
with the $G$-poset with one element.

For lack of reference, we include a proof of the following basic fact.

\begin{lemma}
\label{lemma-induce-lefschetz}
Let $H \subset G$ be an inclusion of finite groups, and let $P$ be an $H$-poset.
Then we have the equality $\Ind{H}{G}\Lef_P = \Lef_{\Ind{H}{G}P}$ in $\Bu(G)$.
\end{lemma}
\begin{proof}
The $G$-poset $\Ind{H}{G}P$ consists of the set of equivalence classes
of pairs $(g, x)$ with $g \in G$ and $x \in P$. The ordering on $\Ind{H}{G}P$
is given by the relations $(g, x) \leq (g, x')$ for all $g \in G$ and
all relations $x \leq x'$ in $P$. The equivalence class containing $(g, x)$
is the set of elements on the form $(gh, h^{-1}x)$ for $h\in H$.
Fix a natural number $i$. The $G$-set $\Ind{H}{G}\Sd_iP$
is the set of equivalence classes of pairs $(g, x_0 < \cdots < x_i)$,
with the equivalence relation defined in the obvious way. We have a
natural $G$-equivariant map
$$
\eta_i\colon \Ind{H}{G}\Sd_iP \to \Sd_i\Ind{H}{G}P
$$
taking an element $(g, x_0 < \cdots < x_i)$ to the chain
$(g, x_0) < \cdots < (g, x_i)$. It is clear that $\eta_i$ is injective.
Furthermore, any chain $(g_0, x_0) < \cdots < (g_i, x_i)$ can be written
on the equivalent form $(g_0, x_0) < (g_0, x'_1) <\cdots < (g_0, x'_i)$,
which proves that it lifts to an element in $\Ind{H}{G}\Sd_iP$.
Hence $\eta_i$ is an isomorphism, which proves the lemma.
\end{proof}

%% file: varieties.tex
\section*{The exponential function}
Inspired by the corresponding construction for Burnside rings
\cite{bouc1992}, we introduce an exponential function on the Grothendieck
ring of varieties. We define the group
$$
\pK(\cVar_k) = \prod_{i \geq 0} \K0(\cgVar{\Sigma_i}_k).
$$
The elements of $\pK(\cVar_k)$ are denoted as formal power series
$\sum_{i \geq 0} a_iT^i$ in the symbol $T$. Note that each of the
coefficients $a_i$ lie in a different group $\K0(\cgVar{\Sigma_i}_k)$.
The group $\pK(\cVar_k)$ is given a multiplicative structure, via the binary
operation $\ast$, which is defined by
$$
\left(\sum_{i \geq 0} a_iT^i\right)\ast\left(\sum_{i \geq 0} b_iT^i\right) 
= \sum_{n \geq 0} \left(\sum_{i+j = n}
\Ind{\Sigma_i\times\Sigma_j}{\Sigma_n}(a_i \boxtimes b_j)\right)T^n.
$$
This gives $\pK(\cVar_k)$ the structure of a ring. We let $\pU(\cVar_k)$
denote the subset of $\pK(\cVar_k)$ of elements of the form
$\sum_{i \geq 0} a_iT^i$ with $a_0 = 1$.   
\begin{prop}
The operation $\ast$ on the group $\pK(\cVar_k)$ gives $\pK(\cVar_k)$ the
structure of a commutative ring with identity. The identity is given by
$1 \in \K0(\cVar_k)$. Moreover, the subset $\pU(\cVar_k) \subset \pK(\cVar_k)$
is a group under the operation $\ast$.
\end{prop}
\begin{proof}
Recall that the operation $\Ind{H}{G}$ is linear for any inclusion $H
\to G$ of groups, and that the operation $\boxtimes$ is bilinear and symmetric.
From this, it follows easily that the operation $\ast$ is commutative
and distributes over addition. Next, we note that we have the identities
$$
\Ind{\Sigma_i\times\Sigma_{j+l}}{\Sigma_n}\left(a\boxtimes
\Ind{\Sigma_j\times\Sigma_l}{\Sigma_{j + l}}\left(b\boxtimes c\right)\right) =
\Ind{\Sigma_i\times\Sigma_j\times\Sigma_l}{\Sigma_n}\left(a\boxtimes
b\boxtimes c\right)
$$
for $a \in \K0(\cgVar{\Sigma_i}_k)$, $b \in \K0(\cgVar{\Sigma_j}_k)$
and~$c \in \K0(\cgVar{\Sigma_l}_k)$ with $i + j + l = n$. From this and the
symmetry of $\boxtimes$, it follows that the product of the three elements
$\sum a_iT^i$, $\sum b_iT^i$ and~$\sum c_iT^i$ can be written
$$
\sum_{n\geq 0}\left(\sum_{i+j+l =
n}\Ind{\Sigma_i\times\Sigma_j\times\Sigma_l}{\Sigma_n}
\left(a_i\boxtimes b_j\boxtimes c_l\right)\right)T^n
$$
regardless of the grouping of the parentheses, which proves associativity.
That $1 \in \K0(\cVar_k)$ is the identity, follows from the trivial
identity
$
\Ind{\Sigma_0\times\Sigma_n}{\Sigma_n} (1\boxtimes a) = a
$
for $a \in \K0(\cgVar{\Sigma_n}_k)$.

The set $\pU(\cVar_k)$ is clearly closed under multiplication.
Let $\sum a_i T^i$ be an element of $\pK(\cVar_k)$ with $a_0 = 1$. We wish to
construct an inverse $\sum b_j T^j$ with $b_0 = 1$. But this can be done
recursively, in exactly the same way as for usual power series, by defining
$b_n$ for each $n > 0$ such that the expression
$$
\sum_{i+j = n}\Ind{\Sigma_i\times\Sigma_j}{\Sigma_n}(a_i\boxtimes b_j)
$$
vanishes. Note that this can be done since we have
$\Ind{\Sigma_0\times\Sigma_n}{\Sigma_n}(a_0\boxtimes b_n) = b_n$ for all
$n > 0$ by the assumption on $a_0$, which allows us to solve for $b_n$
in each step of the recursion.
\end{proof}

Next, we explore the power operations which were mentioned in the
introduction. Let $G$ be a finite group and $S$ a finite $G$-set with $n$
elements. Recall that we defined the $G$-variety $X^S$ as the product
of $n$ copies of $X$ together with the natural $G$-action given by permuting the
factors of the product. We wish to examine its class $\{X^S\}$ in the
equivariant Grothendieck ring $\K0(\cgVar{G}_k)$.
First we note that $X^S$ can be obtained from the variety $X^{[n]}$ via
restriction along the group homomorphism $G \to \Sigma_n$ corresponding to the
action on $S$. Hence it is enough to consider the power operations for the
$\Sigma_n$-sets $[n]$ for various natural numbers $n$. Before we continue, we
need some basic results on powers of closed immersions and vector bundles.
\begin{prop}
\label{prop-power-prop}
Let $Z$ be a closed subvariety of $X$ in $\cVar_k$. Fix a natural number $n$,
and let $X_i$ denote the $\Sigma_n$-invariant closed subvariety of $X^{[n]}$
defined by requiring at least $j = n-i$ of the coordinates to lie in $Z$.
Then the sequence
$$
\emptyset = X_{-1}\subset Z^{[n]} = X_0 \subset \cdots \subset X_n = X^{[n]}
$$
is a filtration of closed immersions such that the stratum
$X_i\setminus X_{i-1}$ is isomorphic to
$\Ind{\Sigma_i\times \Sigma_j}{\Sigma_n}(X\setminus Z)^{[i]}\boxtimes Z^{[j]}$
for $0 \leq i \leq n$. In particular, we have the relation
$$
\left\{X^{[n]}\right\}
=
\sum_{i+j = n}\Ind{\Sigma_i\times\Sigma_j}{\Sigma_n}\left(
  \left\{(X\setminus Z)^{[i]}\right\}\boxtimes \left\{Z^{[j]}\right\}\right)
$$
in $\K0(\cgVar{\Sigma_n}_k)$. 
\end{prop}
\begin{proof}
By the Yoneda Lemma, we may assume that $Z \subset X$ is an inclusion of
sheaves, and by standard arguments, we reduce to the case when $Z \subset X$ is
an inclusion of sets. Consider the following diagram:
$$
\xymatrix{
Z^{[i]}\boxtimes(X\setminus Z)^{[j]} \ar[r]^\eta \ar[dr]_f& 
  \Res{\Sigma_i\times\Sigma_j}{\Sigma_n} X_i\setminus X_{i-1}
  \ar[d]^{\Res{\Sigma_i\times\Sigma_j}{\Sigma_n}f^\natural}\\
& \Res{\Sigma_i\times\Sigma_j}{\Sigma_n} S
}
$$
Here $\eta$ is the $\Sigma_i\times \Sigma_j$-equivariant inclusion
obtained by restricting the identity map on $X^{[n]}$, and $f$ is an
arbitrary $\Sigma_i\times\Sigma_j$-equivariant function to an arbitrary
$\Sigma_n$-set $S$.
By the universal property of induction, it is enough to show that there exists a
unique $\Sigma_n$-equivariant function $f^\natural$ fitting into the diagram. As
a candidate, we take $f^\natural$ to be the function defined by
$$
x \mapsto \tau f(\tau^{-1}x)
$$
where $\tau \in \Sigma_n$ is a permutation such that the first $i$
coordinates of $\tau^{-1} x$ lie in $Z$. The permutation $\tau$ is
unique up to multiplication from the right by an element in
$\Sigma_i\times\Sigma_j$, so $f^\natural$ is well-defined by
$\Sigma_i\times\Sigma_j$-equivariance of $f$. Furthermore, the identities
$$
f^\natural(\sigma x) = \sigma\tau f(\tau^{-1}\sigma^{-1}\sigma x) = \sigma
f^\natural(x), 
$$
where $\sigma\in \Sigma_n$ is arbitrary and $\tau$ is chosen as above,
show that $f^{\natural}$ is $\Sigma_n$-equivariant.
Since the image of $\eta$ generates $X_i\setminus X_{i-1}$ as a $\Sigma_n$-set,
the map $f^\natural$ is also unique with respect to the desired property, which
establishes the isomorphism. The statement about the equivariant Grothendieck
ring follows by summing over the strata.
\end{proof}

Instead of just studying the operation $(-)^{[n]}$ for one fixed $n$,
it is beneficial to study the operations for all $n$ simultaneously.
We do this by using the ring $\pK(\cVar_k)$. Given a variety $X$, we consider
the element  $\sum \{X^{[i]}\}T^i$ in $\pK(\cVar_k)$. Since $\{X^{[0]}\} = 1$,
this element lies in $\pU(\cVar_k)$. Let $Z \subset X$ be a closed subvariety
of $X$. By Proposition~\ref{prop-power-prop} and the definition of
multiplication in the group $\pU(\cVar_k)$, we get the identity
$$
\sum_{n \geq 0}\left\{X^{[n]}\right\}T^n = 
\left(\sum_{i \geq 0} \left\{(X\setminus Z)^{[i]}\right\}T^i\right)
\ast
\left(\sum_{j \geq 0} \left\{Z^{[j]}\right\}T^j\right).
$$
Hence the association $X \mapsto \sum \{X^{[i]}\}T^i$ respects the defining
relations for the Grothendieck group of varieties, and we get a unique group
homomorphism from $\K0(\cVar_k)$ to $\pU(\cVar_k)$ satisfying $\{X\} \mapsto
\sum \left\{X^{[i]}\right\}T^i$.

\begin{definition}
\label{def-exponential}
The group homomorphism described above is called the  \term{exponential
function} and is denoted by
$
\exp\colon \K0(\cVar_k) \to \pU(\cVar_k).
$
\end{definition}

\noindent
The exponential function gives us a natural way to extend the power operation
to the Grothendieck ring $\K0(\cVar_k)$.

\begin{definition}
Let $x \in \K0(\cVar_k)$, and let $G$ be a finite group acting on a set $S$
with $n$ elements. The \term{power of $x$ by $S$} is defined as the element
$$
x^{S} := \Res{G}{\Sigma_n}c_n
$$
in $\K0(\cgVar{G}_k)$, where $c_n$ denotes the coefficient for $T^n$ in
$\exp(x)$.
\end{definition}

The results of this section prove the Binomial Theorem stated in the
introduction. Indeed, the formula for addition follows directly from the
definition of multiplication in the ring $\pK(\cVar_k)$. The statement about
vector bundles is the following proposition.

\begin{prop}
\label{prop-bundle}
Let $E \to X$ be a vector bundle of rank $d$ in $\cVar_k$. Then
$E^{[n]} \to X^{[n]}$ is a vector bundle of rank $nd$ in $\cgVar{\Sigma_n}_k$.
In particular, we have $\left\{E^{[n]}\right\} = \LL^{nd}\cdot \left\{X^{[n]}\right\}$
in $\K0(\cgVar{\Sigma_n}_k)$.
\end{prop}
\begin{proof}
The fact that $E^{[n]} \to X^{[n]}$ is a rank $nd$ vector bundle is well-known
and easy to prove. See for instance \cite[Lemma~2.4]{ekedahl2009fg} for a
proof in the more general context of algebraic stacks. The statement about
the class in the equivariant Grothendieck ring is simply the defining
relation~R2.
\end{proof}

%% file: tori.tex
\section*{Classes of quasi-split tori}
The methods introduced in the previous section give us 
a convenient way to compute the class of a quasi-split torus in the
Grothendieck ring of varieties. As explained in the introduction, it is enough
to consider the universal case.
\begin{prop}
\label{prop-class-torus-raw}
The class of the universal quasi-split torus $\GGm^{[n]}$ of rank $n$
is given by the expression
$$
\sum_{i = 0}^n
\Ind{\Sigma_i\times\Sigma_{n-i}}{\Sigma_n}\left((-1)^{[i]}\boxtimes
1\right) 
\cdot\LL^{n-i}
$$
in the equivariant Grothendieck ring $\K0(\cgVar{\Sigma_n}_k)$. 
\end{prop}
\begin{proof}
Since the class of $\GGm$ is $\LL-1$ in $\K0(\cVar_k)$, the class of
$\GGm^{[n]}$ is given by
$$
\sum_{i = 0}^n
\Ind{\Sigma_i\times\Sigma_{n-i}}{\Sigma_n}\left((-1)^{[i]}\boxtimes
\LL^{[n-i]}\right)
$$
according to the Binomial Theorem. From Proposition~\ref{prop-bundle},
it follows that $\LL^{[n-i]}$ equals $\LL^{n-i}$. Hence the result follows
from the $\ZZ[\LL]$-linearity of the operations $\boxtimes$ and
$\Ind{\Sigma_i\times\Sigma_{n-i}}{\Sigma_n}$.
\end{proof}
\noindent
Theorem~\ref{thm-equiv-torus} and its corollaries will now follow provided
that we establish the identity
$$
\label{eqn-coefficients}
(-1)^i\lambda^i([n])
= \Ind{\Sigma_i\times\Sigma_{n-i}}{\Sigma_n}\left((-1)^{[i]}\boxtimes 1\right)
$$
for the coefficients in the expression for the class of $\GGm^{[n]}$ given in
Proposition~\ref{prop-class-torus-raw}. This computation can be done entirely
in the Burnside ring, which allows us to use the methods introduced by Bouc in
\cite{bouc1992}. We briefly recapitulate the parts of the theory that we
need.

The ring $\pK(\cVar_k)$ has an exact analog for Burnside rings, namely the ring
$$
\pBu = \prod_{i \geq 0} \Bu(\Sigma_i).
$$
We think of elements of $\pBu$ as formal power series $\sum a_i T^i$
and define the multiplication using the same formula as for $\pK(\cVar_k)$.
We also have an exponential function $\exp\colon\ZZ \to \pBu$ induced by taking
a finite set $S$ with $n$ elements to the power series $\sum
\left\{S^{[i]}\right\}T^i$. This operation extends to all of $\ZZ$.

The group homomorphisms $\Bu(\Sigma_i) \to \K0(\cgVar{\Sigma_i}_k)$ of the
coefficient rings combine to a group homomorphism
$\iota\colon\pBu \to \pK(\cVar_k)$.
Since the induction and outer products commute with the maps
$\Bu(\Sigma_i) \to \K0(\cgVar{\Sigma_i}_k)$, we see that $\iota$ is in fact
a ring homomorphism. It also respects the exponential functions, which shows
that we can compute the powers $n^{[i]}$ in the Burnside rings for all
integers $n$ and then map the result to $\K0(\cgVar{\Sigma_i}_k)$.

Powers in the Burnside rings can be computed by means of the Lefschetz
invariant. Given a poset $P$ and a natural number $i$, we define
$P^{[i]}$ as the $i$-fold product of $P$ with itself. This set has a
structure of $\Sigma_i$-poset, with the $\Sigma_i$-action permuting the
factors and the ordering given by coordinate-wise comparison. More
specifically, we have $(x_1, \ldots, x_i) \leq (x'_1, \ldots, x'_i)$
provided that $x_j \leq x_j'$ for $1 \leq j \leq i$. Given an integer
$n$ and a poset $P$ such that $n = \Lef_P$, we have $n^{[i]} = \Lef_{P^{[i]}}$
for all natural numbers $i$.

In order to compute the powers of $-1$, we introduce the following notation.
Let $G$ be a finite group and $S$ a finite $G$-set. Given an integer $i$,
we let $\Omega_{\leq i}S$ denote the set of non-empty subsets of $S$
with cardinality less or equal to $i$. The set $\Omega_{\leq i}S$ has
$G$-poset structure with its $G$-action inherited from $S$ in the obvious
way, and the ordering given by inclusion. We use the symbol $S_+$ to denote
the disjoint union of the $G$-set $S$ with the one-point
$G$-set $\{\bullet\}$.

\begin{lemma}
\label{lemma-minus-one}
We have the identity $(-1)^{[n]} = \rLef_{\Omega_{\leq n}([n]_+)}$ in
$\K0(\cgVar{\Sigma_n}_k)$ for each natural number $n$.
\end{lemma}
\begin{proof}
By \cite[Lemme 4]{bouc1992}, we have the identity $(-1)^{[n]} =
-\rli_{2^{\ast n}}$. The $\Sigma_n$-poset $2^{\ast n}$ can be
described as the set $F$ of functions from $[n]$ to $\{a, b, \infty\}$
which are not constant $\infty$. The $\Sigma_n$-action is given by
$(\sigma\cdot\varphi)(x) = \varphi(\sigma^{-1}x)$ for $\sigma\in\Sigma_n$.
The set $\{a, b, \infty\}$ has an ordering given by $a < \infty$, $b < \infty$,
and the ordering on the set $F$ is given by $f \leq g$ if $f(x) \leq g(x)$
for all $x \in [n]$. This gives $F$ a $\Sigma_n$-poset structure. Consider
the map
$
f\colon F \to \Omega_{\leq n}([n]_+)
$
defined by
$$
\varphi \mapsto
\left\{
\begin{array}{ll}
\varphi^{-1}(a) & \text{if } b\not\in \Im\varphi, \\
\varphi^{-1}(a)\cup \{\bullet\} & \text{otherwise.}\\
\end{array}
\right.
$$
We also have a map $g$ in the other direction, which takes a subset $U$
to the function $\varphi_U$ defined by
$$
\varphi_U(x)=
\left\{
\begin{array}{ll}
a & \text{if } x \in U, \\
b & \text{if } x \not\in U \text{ and } \bullet \in U,\\
\infty & \text{otherwise.}\\
\end{array}
\right.
$$
Both the maps $f$ and $g$ are $\Sigma_n$-equivariant and order-reversing. 
One easily verifies that $h\circ g$ is the identity and that $g\circ h$
is comparable with the identity. Hence $\Omega_{\leq n}([n]_+)$ is homotopy
equivalent to $F$ considered with its opposite ordering. Since reversing
the ordering of a $G$-poset clearly does not change its Lefschetz invariant,
the statement in the lemma follows.
\end{proof}

Next, we need to investigate how the $G$-posets $\Omega_{\leq i}(S)$
transform under induction. Fix natural numbers $i' \leq i$. We let $\Omega_{\leq i', i}(S)$
denote the set of flags $U \subset V$ of non-empty subsets of $S$, where $U$
has at most $i'$ elements and $V$ has exactly $i$ elements.
These flags are ordered by level-wise inclusion, and we have a natural action
of $G$ on $\Omega_{\leq i', i}(S)$ respecting the ordering.

\begin{lemma}
\label{lemma-induce-powerset}
Let $i' \leq i \leq n$ be natural numbers and define $j = n - i$. We consider
$\Omega_{\leq i'}([i])$ a $\Sigma_i\times\Sigma_{j}$-poset by letting the
second factor act trivially. Then
$$
\Ind{\Sigma_i\times\Sigma_{j}}{\Sigma_n}\Omega_{\leq i'}([i]) \simeq
\Omega_{\leq i', i}([n])
$$
as $\Sigma_n$-posets.
\end{lemma}
\begin{proof}
We use a similar argument as in the proof of Proposition~\ref{prop-power-prop}.
Consider the $\Sigma_i\times\Sigma_{j}$-poset morphism
$$
\eta\colon\Omega_{\leq i'}([i]) \to
\Res{\Sigma_i\times\Sigma_{j}}{\Sigma_n} \Omega_{\leq i', i}([n])
$$
that takes a subset $U$ of $[i]$ to the flag $U \subset [i]$. We prove that any
$\Sigma_i\times\Sigma_{j}$-poset morphism
$f\colon \Omega_{\leq i'}([i]) \to \Res{\Sigma_i\times\Sigma_{j}}{\Sigma_n} P$,
where $P$ is an arbitrary $\Sigma_n$-poset, uniquely extends to a
$\Sigma_n$-poset morphism $f^\natural\colon\Omega_{\leq i',i}([n]) \to P$ via
$\eta$. As a candidate, we take the function
$$
f^\natural(U \subset V) = \tau f(\tau^{-1} U)
$$
where $\tau$ is a permutation such that $\tau^{-1}V = [i]$. The
permutation $\tau$ is unique up to right multiplication by a permutation
in $\Sigma_i\times\Sigma_{j}$. Hence $f^\natural$ is well-defined by
$\Sigma_i\times\Sigma_{j}$-equivariance of $f$. For an arbitrary element
$\sigma \in \Sigma_n$, we have the equalities
$$
\sigma f^\natural(U \subset V)
= \sigma\tau f(\tau^{-1} U)
= \sigma\tau f((\tau^{-1} \sigma^{-1}) \sigma U)
= f^\natural(\sigma(U\subset V)).
$$
Here the last step follows from the identity
$(\tau^{-1} \sigma^{-1})\sigma V = [i]$.
Since $f^\natural$ clearly is order-preserving, this shows that $f^\natural$
is a $\Sigma_n$-poset morphism. Finally, since the image
of $\eta$ generates $\Omega_{\leq i', i}([n])$ as a $\Sigma_n$-set,
the extension $f^\natural$ is a unique with respect to the desired
property, which concludes the proof.
\end{proof}

Finally, we need an expression for the lambda operations in terms of Lefschetz
invariants of equivariant posets. Such expressions have been computed by Bouc
and Rökaeus in \cite{br2009}. We state the result together with an auxiliary
result from the same article.

\begin{lemma}[Bouc--Rökaeus]
\label{lemma-br}
Let $G$ be a finite group and $S$ a finite $G$-set. Then we have the
following identities in the Burnside ring $\Bu(G)$:
\begin{enumerate}
\item[(a)] $\lambda^n(S) = (-1)^{n-1}\rLef_{\Omega_{\leq n}(S_+)}$
\item[(b)] $\rLef_{\Omega_{\leq n}(S_+)} = \rLef_{\Omega_{\leq n}(S)} -
\rLef_{\Omega_{\leq n-1}(S)}$
\end{enumerate}
\end{lemma}

Now we are ready to combine the results and prove the main theorem about the
expression for the class of the universal quasi-split torus.

\begin{proof}[Proof of Theorem~\ref{thm-equiv-torus}]
By Lemma~\ref{lemma-br}~(b), we have the identities
$$
\rLef_{\Omega_{\leq i}([i]_+)}
= \rLef_{\Omega_{\leq i}([i])} - \rLef_{\Omega_{\leq i - 1}([i])}
= \Lef_{\Omega_{\leq i}([i])} - \Lef_{\Omega_{\leq i - 1}([i])}.
$$
Using that taking the Lefschetz invariant commutes with induction as proved
in Lemma~\ref{lemma-induce-lefschetz}, together with the explicit description
from Lemma~\ref{lemma-induce-powerset} of the induced posets involved, we get
$$
\Ind{\Sigma_i\times\Sigma_{n-i}}{\Sigma_n}\rli_{\Omega_{\leq i}([i]_+)}
= \li_{\Omega_{\leq i, i}([n])} - \li_{\Omega_{\leq i - 1, i}([n])}.
$$
Next, we note that the element $\Sd_j \Omega_{\leq i, i}([n]) - \Sd_j
\Omega_{\leq i - 1, i}([n])$ in $\Bu(\Sigma_n)$ is effective. It is
represented by the $\Sigma_n$-set of length $j$ chains of flags $V \subset U$
in $[n]$, with $V$ non-empty and $U$ having cardinality $i$, whose
maximal element satisfies $V = U$. But it is obvious that the set $U$ is
redundant in this description. The set is equivariantly isomorphic to the set of
chains of non-empty subsets of $[n]$ with the maximal subset of cardinality $i$,
which has the class $\Sd_j \Omega_{\leq i}([n]) - \Sd_j \Omega_{\leq i -
1}([n])$.
By using Lemma~\ref{lemma-br}~(b) again, we therefore get 
$$
\Ind{\Sigma_i\times\Sigma_{n-i}}{\Sigma_n}\rLef_{\Omega_{\leq i}([i]_+)}
= \Lef_{\Omega_{\leq i}([n])} - \Lef_{\Omega_{\leq i - 1}([n])}
= \rLef_{\Omega_{\leq i}([n]_+)}.
$$
By using the expression $(-1)^{[i]} = -\rLef_{\Omega_{\leq i}([i]_+)}$
derived in Lemma~\ref{lemma-minus-one}, we get
$$
\Ind{\Sigma_i\times\Sigma_{n-i}}{\Sigma_n}\left((-1)^{[i]}\boxtimes 1\right)
= -\rLef_{\Omega_{\leq i}([n]_+)}.
$$
By applying Lemma~\ref{lemma-br}~(a) to the right hand side of this identity,
we see that the coefficients in the formula of
Proposition~\ref{prop-class-torus-raw} coincides with the coefficients in
Theorem~\ref{thm-equiv-torus}.
\end{proof}
\begin{remark}
Note that we have only used a very special case of Bouc's construction
of the exponential function. He considers the more general ring
$\pBu(G) = \prod_{i \geq 0} \Bu(G \wr\Sigma_i)$ where $G$ is an arbitrary
finite group. Here $G \wr\Sigma_i$ denotes the wreath product of $G$ by
$\Sigma_i$. In particular, there are power operations $(-)^{[i]}\colon\Bu(G) \to
\Bu(G \wr\Sigma_i)$ for all $i\in \NN$. A similar construction could be made
for equivariant Grothendieck rings of varieties. In this setting, one could
compute the class of the Weil restriction along a finite étale map for more
general varieties $X$. More precisely, the class $\{X\}$ could be a polynomial
in $\LL$ with coefficients in finite étale classes.
\end{remark}